\documentclass[twoside]{amsart}
\usepackage{latexsym}
\usepackage{enumitem}
\usepackage{amssymb,amsmath,amsopn}
\usepackage[dvips]{graphicx}   
\usepackage{color,epsfig}      
\usepackage{tikz} 
\usepackage{subcaption} 
\usepackage{array}
\usepackage{hyperref} \usepackage{mathtools}

\theoremstyle{plain}
\newtheorem{theorem}{Theorem}[section]
\newtheorem{lemma}[theorem]{Lemma}
\newtheorem{conjecture}{Conjecture}
\newtheorem{corollary}[theorem]{Corollary}

\newtheorem{remark}[theorem]{Remark}
\theoremstyle{definition}
\newtheorem{definition}[theorem]{Definition}
\numberwithin{equation}{section}

\renewcommand{\Re}{\mathbb R}

\renewcommand{\S}{\mathbb S}

\newcommand{\B}{\mathbf{B}}

\newcommand{\BB}{\mathbf B}

\newcommand{\HHH}{\mathbb{H}}
\newcommand{\Sph}{\mathbb{S}}
\newcommand{\MM}{\mathbb{M}}
\newcommand{\V}{\mathbb{V}}

\DeclareMathOperator{\arctanh}{arctanh}

\DeclareMathOperator{\inter}{int}
\DeclareMathOperator{\bd}{bd}

\DeclareMathOperator{\conv}{conv}

\DeclareMathOperator{\vol}{vol}

\DeclareMathOperator{\arcosh}{arcosh}
\begin{document}

\title[Equilibria]{Equilibria in non-Euclidean geometries}

\author[Z. L\'angi]{Zsolt L\'angi}
\author[S. Wang]{Shanshan Wang}

\address{Zsolt L\'angi, Bolyai Institute, University of Szeged\\
Aradi v\'ertan\'uk tere 1, H-6720 Szeged, Hungary, and\\
Alfr\'ed R\'enyi Institute of Mathematics\\
Re\'altanoda utca 13-15, H-1053, Budapest, Hungary}
\email{zlangi@server.math.u-szeged.hu}
\address{Shanshan Wang, Bolyai Institute, University of Szeged\\
Aradi v\'ertan\'uk tere 1, H-6720 Szeged, Hungary}
\email{shanshanwang@server.math.u-szeged.hu}

\thanks{Partially supported by the ERC Advanced Grant ``ERMiD'' and the National Research, Development and Innovation Office, NKFI, K-147544 grant, and the Project no. TKP2021-NVA-09 with the support provided by the Ministry of Innovation and Technology of Hungary from the National Research, Development and Innovation Fund and financed under the TKP2021-NVA funding scheme.}


\begin{abstract}
In this paper, extending the work of Gal'perin (Comm. Math. Phys. 154: 63-84, 1993), we investigate generalizations of the concepts of centroids and static equilibrium points of a convex body in spherical, hyperbolic and normed spaces. In addition, we examine the minimum number of equilibrium points a $2$- or $3$-dimensional convex body can have in these spaces. In particular, we show that every plane convex body in any of these spaces has at least four equilibrium points, and that there are mono-monostatic convex bodies in $3$-dimensional spherical, hyperbolic, and certain normed spaces. Our results are generalizations of results of Domokos, Papadopoulos and Ruina (J. Elasticity 36: 59-66, 1994), and V\'arkonyi and Domokos (J. Nonlinear Sci. 16: 255-281, 2006) for convex bodies in Euclidean space.
\end{abstract}

\maketitle

\section{Introduction}\label{sec:intro}

Centroids and static equilibrium points of convex bodies have been in the focus of research since the work of Archimedes, whose results were used in naval design even in the 18th century \cite{Archimedes1}. These concepts appear not only in physics, but also in engineering \cite{robotics, varkonyirobotics}, geology \cite{pebbles, korvin}, biology \cite{turtles, turtles2}, medicine \cite{pill} and in many other disciplines. Among the classical problems in pure mathematics related to these concepts we can mention problems related to floating bodies ( see e.g. \cite{scottish, ryabogin, huang}), Gr\"unbaum's inequality \cite{grunbaum} and its variants, and the Busemann-Petty centroid inequality \cite{petty} and its variants. There are also many mathematical results related to the number of equilibrium points of convex bodies. As examples, we may mention the construction of a double-tipping tetrahedron by Heppes \cite{Heppes}, or the monostable polyhedron of Conway and Guy in \cite{GG}. Recent solutions of two problems on monostable polyhedra \cite{Langi22} by the first-named author, proposed by Conway and Guy \cite{GG} in 1969  are also worth mentioning here. As for general convex bodies, a result of Domokos, Papadopoulos and Ruina \cite{DPR} states that a nondegenerate convex body in the Euclidean plane has at least four equilibria, whereas the construction of a mono-monostatic convex body in \cite{VD} by V\'arkonyi and Domokos, called G\"omb\"oc, shows that the same statement does not hold for $3$-dimensional convex bodies. We note that the latter body, the existence of which was conjectured by Arnold \cite{VD2}, is well known outside the mathematics and physics community and has found applications in many other disciplines\footnote{For example, it was chosen to symbolize monetary stability by the National Bank of Hungary, and also as the insignia for the Steven Smale Prize by the Society for the Foundations of Computational Mathematics}. For more information on the mathematical aspects of centroids and equilibria, the interested reader is referred to the recent survey \cite{langivarkonyi}.

In the past century, numerous attempts have been made to generalize the notions of centroids to non-Euclidean geometries, mostly involving spherical and hyperbolic spaces. In particular, centroids of spherical triangles have been investigated by Fog \cite{Fog} and Fabricius-Bjerre \cite{Bjerre} in the 1940s. Their definition uses the model of spherical space as a hypersurface embedded in a Euclidean space (for such an approach, see also \cite{Brock}). In 1993, Gal'perin \cite{Galperin} gave a thorough investigation of possible definitions of centroids of systems of points in spaces of constant curvature, most importantly in spherical and hyperbolic space.  His results were recently generalized for $k$-dimensional manifolds in these spaces in \cite{CC}. A recent paper of Besau et al. \cite{BHPS} independently defines the centroid of a spherical set in the same way.

The main goal of this paper is to extend the above investigation to equilibrium properties of convex bodies in non-Euclidean spaces, more specifically in spherical, hyperbolic and normed spaces. In particular, we investigate how the results in \cite{DPR} and \cite{VD} can be generalized for convex bodies in these spaces. During this investigation, we always imagine a $d$-dimensional normed space to be equipped with an underlying Euclidean space.

Our two main results are as follows.

\begin{theorem}\label{thm:2d}
Let $\V^2$ be a plane of constant curvature, or a normed plane with a smooth and strictly convex unit disk. Let $K$ be a plane convex body in $\mathbb{V}$. Then $K$ has at least four equilibrium points.
\end{theorem}

We note that, unlike in the paper \cite{DPR}, we do not make regularity assumptions on $K$. Thus, our result (in this sense) is a generalization of the result in \cite{DPR} even for convex bodies in the Euclidean plane.

Before the next theorem, recall that for any metric space with metric $d(\cdot,\cdot)$, the Hausdorff distance of two compact sets $A,B$ is defined as
\begin{multline*}
d_H(A,B)= \inf \{ d \geq 0: \hbox{ for any } a \in A \hbox{ there is some } b \in B \hbox{ with } d(a,b) \leq d,\\
\hbox{ and } \hbox{ for any } b \in B \hbox{ there is some } a \in A \hbox{ with } d(a,b) \leq d \}.
\end{multline*}
In Theorem~\ref{thm:3d}, for $d \geq 3$, we call a $d$-dimensional normed space \emph{rotationally symmetric} if there is a  $(d-2)$-dimensional linear subspace $L$ in it such that the norm is invariant under any rotation around $L$.

\begin{theorem}\label{thm:3d}
Let $\V^3$ be a
\begin{enumerate}
\item[(a)] a $3$-dimensional space of constant curvature; or
\item[(b)] a $3$-dimensional rotationally symmetric normed space with a unit ball $M$. In this case we assume that $K$ has $C^2$-class boundary and strictly positive Gaussian curvature everywhere.
\end{enumerate}
Then, for every closed ball $\BB$ in $\V^3$ and every $\varepsilon > 0$ there is a mono-monostatic convex body $K$ in $\V^3$ with $d_H(\BB,K) \leq \varepsilon$.
\end{theorem}

We note that our condition in (b) about the rotational symmetry of the space seems to be technical. We make the following conjecture.

\begin{conjecture}\label{conj:normedGomboc}
In any $3$-dimensional normed space $\MM^3$ whose unit ball has $C^2$-class differentiable boundary and strictly positive Gaussian curvature, there is a mono-monostatic convex body.
\end{conjecture}




The structure of the paper is as follows. In Section~\ref{sec:prelim} we present the necessary background for the proofs and discuss how to define centroids and equilibria in non-Euclidean geometries. In Section~\ref{sec:2d} we prove Theorem~\ref{thm:2d}. Finally, in Section~\ref{sec:3d} we prove Theorem~\ref{thm:3d}.

We note that whereas the main result of this paper is the generalization of the Euclidean results in \cite{DPR, VD} for non-Euclidean geometries, it seems reasonable to assume that the tools and approach presented in this paper permits to find non-Euclidean variants of many results about the equilibrium properties of convex bodies in these geometries.

\section{Preliminaries}\label{sec:prelim}

In this paper, we denote the $d$-dimensional Euclidean space by $\Re^d$. For any two points $p,q \in \Re^d$, we denote by $[p,q]$ the closed segment with endpoints $p,q$, and by $|p|$ the Euclidean norm of $p$. The \emph{(Euclidean) distance} of $p,q \in \Re^d$ is $d_E(p,q)=|q-p|$. We denote the closed $d$-dimensional unit ball centered at the origin $o$ by $\B^d$, and its boundary by $\Sph^{d-1}$. Furthermore, for any set $S \subset \Re^3$ we let $\conv (S)$ denote the convex hull of $S$. By a convex body we mean a compact, convex set with nonempty interior.
In this space, $d$-dimensional Lebesgue measure is called \emph{(Euclidean) volume}, denoted by $\vol_d^E(\cdot)$. To distinguish it from non-Euclidean volumes, in the proofs we also use the notation $\lambda_d(\cdot)$ for it. In the paper we also use the standard notation $\inter(\cdot)$, $\bd(\cdot)$ for the interior and the boundary of a set.

The classical model of the $d$-dimensional spherical space is the set $\Sph^d$ defined as the unit sphere of $\Re^{d+1}$. Here, the spherical distance $d_S(p,q)$ of two points $p,q \in \Sph^d$ is the angle between the rays starting at the origin $o$ and passing through $p,q$, that is, $d_S(p,q)= \arccos \langle p,q \rangle$, where $\langle \cdot, \cdot \rangle$ is the standard inner product of $p,q$ in $\Re^{d+1}$. This space is a $d$-dimensional Riemannian manifold. The volume induced by this structure is called \emph{spherical volume}, which we denote as $\vol_d^S(\cdot)$; this concept coincides with $d$-dimensional Hausdorff measure. Geodesic lines in this space are greatcircles in $\Sph^d$. For any $p,q \in \Sph^d$ with $q \neq -p$, the shorter greatcircle arc of $\Sph^d$ connecting them is called the \emph{(spherical) segment} connecting $p,q$. A set $X \subseteq \Sph^d$, contained in an open hemisphere of $\Sph^d$, is called \emph{convex} if for any $p,q \in X$ it contains the segment connecting them. Using this definition we may define (spherical) convex bodies following the Euclidean definition.

The `dual' model, for hyperbolic geometry, is the so-called \emph{hyperboloid model}. In this model, the hyperbolic space $\HHH^d$ is one sheet of a two-sheeted hyperboloid in the $(d+1)$-dimensional Lorenz space $\mathbb{L}^{d+1}$, equipped with the Lorenz inner product $\langle p,q \rangle_L = x_1y_1+ x_2y_2+ \ldots x_dy_d-x_{d+1}y_{d+1}$ for $p=(x_1,x_2,\ldots,x_{d+1}), q=(y_1,y_2,\ldots,y_{d+1}) \in \mathbb{L}^{d+1}$, i.e. with this notation $\HHH^d$ is the set of points $\{ p \in \mathbb{L}^{d+1} : \langle p,p \rangle_L = -1, x_{d+1} > 0 \}$. The hyperbolic distance $d_H(p,q)$ of the points $p,q \in \HHH^d$ is defined as the quantity 
$d_H(p,q)= \arcosh (-\langle p,q \rangle_L)$. The hyperbolic space $\HHH^d$ in this model is a Riemannian submanifold of $\mathbb{L}^{d+1}$. The volume induced by this structure is called \emph{hyperbolic volume}, and we denote it by $\vol_d^H(\cdot)$. Geodesic lines are intersections of the model with planes containing the origin $o$ of $\mathbb{L}^{d+1}$. For any $p,q \in \HHH^d$, the bounded, connected geodesic arc connecting them is called the \emph{(hyperbolic) segment} with endpoints $p,q$. A set $X \subseteq \HHH^d$ is called \emph{convex} if for any $p,q \in X$, $X$ contains the segment connecting them. Using this definition we may define convex bodies as in Euclidean or spherical spaces. We note that other, equivalent models of the $d$-dimensional hyperbolic space are also often used in the literature. One of them, the so-called \emph{projective ball model} is used also in this paper. We describe the properties of this model at the appropriate place in the proofs.

We finish this part with a brief introduction to normed spaces. Consider a $d$-dimensional real vector space $V$. Let $f : V \to \Re$ be a nonnegative function which is
zero only at zero vector $o \in V$, positively $1$-homogeneous, and satisfies the triangle inequality. Then $f$ is called a \emph{norm} and the pair $(V,f)$ a \emph{normed space}; these spaces are commonly used outside mathematics, see e.g. \cite{Moretti}. In the literature the space $V$ is often identified with the Euclidean space $\Re^d$; we follow this approach in this paper. It is well known that the unit ball $M = \{ p \in \Re^d : f(p)\leq 1 \}$ of the space is an $o$-symmetric convex body, and for any $o$-symmetric convex body $M$, the function $f(v) = \inf \{ \lambda > 0 : v \in \lambda M \}$ is a norm with unit ball $M$. This relation induces a bijection between $o$-symmetric, $d$-dimensional convex bodies and $d$-dimensional norms. In the paper, for any such body $M$ we denote the induced norm by $|| \cdot ||_M$, and the corresponding space by $\MM^d$. 
The norm $||\cdot||_M$ is called \emph{strictly convex} if its unit ball $M$ is strictly convex, i.e. if $\bd(M)$ does not contain nondegenerate segments. The norm $||\cdot||_M$ is \emph{smooth} if $M$ is smooth, or in other words, for any point $p \in \bd(M)$ there is a unique supporting hyperplane of $M$ at $p$. This property is equivalent to the property that $\bd(M)$, as a hypersurface in $\Re^d$, is continuously differentiable \cite{Schneider}. Finally, for any $p,q \in \MM^d$, the \emph{(normed) distance} of $p,q$ is the quantity $d_M(p,q)=||q-p||_M$.

Throughout the paper, by a \emph{(closed) ball} (if $d=2$, \emph{disk}) of radius $r$ and center $p$ we denote the set of points of the space at distance at most $r$ from $p$.

For more information on spherical, hyperbolic or normed spaces, the reader is referred to the books \cite{MM, ratcliffe, thompson}, respectively.

\subsection{Centroids in non-Euclidean spaces}\label{subsec:centroids}

Consider a finite system of points $p_1, p_2, \ldots, p_k \in \Re^d$ with assigned masses $m_1, m_2, \ldots, m_k$, respectively. We call the point $p$ with mass $m$, defined by
\[
m = \sum_{i=1}^k m_i, \quad p= \frac{1}{m} \sum_{i=1}^k m_i p_i.
\]
the \emph{weighted centroid} of this system.
This definition satisfies the following properties:
\begin{itemize}
\item[(i)] If we decompose the system into subsystems, and each subsystem is replaced by its weighted centroid, then the centroid of the system does not change.
\item[(ii)] The weighted centroid of the image of a system under any isometry of the space coincides with the image of the weighted centroid of the original system under the same isometry. 
\end{itemize}
Here, we call the point $p$ the \emph{centroid} of the point system.
This definition satisfies the so-called `Euclidean rule of the lever': for two points $p_1,p_2$ with masses $m_1, m_2$, respectively, their weighted centroid is the pair $(p,m)$ with $p \in [p_1,p_2]$ satisfying $||p_1-p|| m_1 = ||p_2-p|| m_2$ and $m=m_1+m_2$.
We note that this rule, combined with the property in (i) makes it possible to define the centroid of a weighted point system by induction on the number of points.
This approach can be generalized for any Borel set with positive measure, in particular, for convex bodies with uniform density, in a straightforward way.

\begin{definition}\label{defn:centroidRed}
Let $K \subset \Re^d$ be a convex body.
Then the \emph{centroid} of $K$ is
\begin{equation}\label{eq:centroidRed}
c(K) = \frac{1}{\vol_d^E(K)} \int_{x \in K} x \, d \, \vol_d^E(x).
\end{equation}
\end{definition}

\begin{remark}\label{rem:leverruleRed}
We note that for any convex body $K \subset \Re^d$, its centroid $c(K)$ is the unique point satisfying the following:
For any oriented hyperplane $H$, we have
\begin{equation}\label{eq:1stmomentRed}
\int_{x \in K} d_E^s(H,x) d \, \lambda_d(x) = 0 \Longleftrightarrow c(K) \in H,
\end{equation}
where $d_E^s(H,x)$ denotes the signed Euclidean distance of $H$ and $x$; the sign is determined by the orientation of $H$. This formula is a consequence of the Euclidean rule of the lever for convex bodies, and the integral in (\ref{eq:1stmomentRed}), denoted by $M_H(K)$, is called the \emph{first moment} of $K$ with respect to $H$.
\end{remark}

The Euclidean definition of weighted centroid as well as centroid make use of the linear structure of the Euclidean space $\Re^d$, and thus, it cannot be applied
in hyperbolic and spherical space. A thorough investigation of possible definitions of weighted centroids of weighted point systems in hyperbolic or spherical sapce was carried out by Gal'perin in \cite{Galperin} in 1993. A simple approach to define the centroid of a spherical set was used e.g. in \cite{Brock, Bjerre, Fog, BHPS}. Here we use a definition from \cite{Galperin} for a finite system of weighted points in an arbitrary $d$-dimensional hypersurface $S$ embedded in $\Re^{d+1}$.

Consider points $p_1, p_2, \ldots, p_k \in S$ with weights $m_1, m_2, \ldots, m_k$, respectively. Assume that $\sum_{i=1}^k m_i p_i \neq o$, and the ray starting at $o$ and passing through this point intersects $M$ in a unique point $p$. Then $p$ is called the \emph{centroid} of the system, and the pair $(p,m)$, where $m$ is defined by the equality $m p = \sum_{i=1}^k m_i p_i$, is called the \emph{weighted centroid} of the system.
We remark that, imagining the $d$-dimensional Euclidean space as the hyperplane $\{ x_{d+1}= 1 \}$ in $\Re^{d+1}$, this definition of the weighted centroid of a weighted point system coincides with the one presented in the beginning of Subsection~\ref{subsec:centroids}.
Furthermore, for the standard model of the $d$-dimensional spherical space $\Sph^d$ as the unit sphere of the space $\Re^{d+1}$ centered at $o$, the above definition of centroid coincides with the classical one in \cite{Brock, Bjerre, Fog, BHPS}.
Based on this, we define the centroid of a convex body in $\Sph^d$ as follows.

\begin{definition}\label{defn:centroidSphd}
Let $K$ be a convex body in the spherical space $\Sph^d$, embedded in $\Re^{d+1}$ as the unit sphere centered at the origin $o$. 
Then the \emph{centroid} of $K$ is
\begin{equation}\label{eq:centroidSphd}
c(K) = \frac{1}{|\int_{x \in K} x \, d \vol_d^S(x)|} \int_{x \in K} x \, d \vol_d^S(x).
\end{equation}
\end{definition}

Regarding Definition~\ref{defn:centroidSphd}, the fact that every spherical convex body in $\Sph^d$ is contained in an open hemisphere of $\Sph^d$ guarantees that $\int_{x \in K} x \, d \vol_d^S(x) \neq o$. The following spherical version of Remark~\ref{rem:leverruleRed} follows from \cite{Galperin} for point systems (see also \cite{BHLL} for the version for convex bodies).

\begin{remark}\label{rem:leverruleSphd}
For any convex body $K \subset \Sph^d$, its centroid $c(K)$ is the unique point satisfying the following:
For any oriented hyperplane ($(d-1)$-dimensional greatsphere) $H$, we have
\begin{equation}\label{eq:1stmomentSphd}
\int_{x \in K} \sin (d_S^s(H,x)) \, d \vol_d^S(x) = 0 \Longleftrightarrow c(K) \in H,
\end{equation}
where $d_S^s(H,x)$ denotes the signed spherical distance of $H$ and $x$; the sign is determined by the orientation of $H$. Here the integral in (\ref{eq:1stmomentSphd}), denoted by $M_H(K)$, is called the \emph{first moment} of $K$ with respect to $H$.
\end{remark}

One may apply the above approach for the hyperbolic space $\HHH^d$ using the hyperboloid model, which was introduced in the beginning of this section. This leads to the following definition.

\begin{definition}\label{defn:centroidHHd}
Let $K$ be a convex body in the hyperbolic space $\HHH^d$, embedded in $\Re^{d+1}$ as the sheet of hyperboloid with equation $x_1^2+x_2^2 \ldots + x_d^2 - x_{d+1}^2 = 1$, $x_{d+1} > 0$. Then the \emph{centroid} of $K$ is
\begin{equation}\label{eq:centroidHHd}
c(K) = \frac{1}{|\int_{x \in K} x \, d \vol_d^H(x)|} \int_{x \in K} x \, d \vol_d^H(x).
\end{equation}
\end{definition}

This definition yields the following hyperbolic variant of Remark~\ref{rem:leverruleRed}.

\begin{remark}\label{rem:leverruleHHHd}
For any convex body $K \subset \HHH^d$, its centroid $c(K)$ is the unique point satisfying the following:
For any oriented hyperplane $H$, we have
\begin{equation}\label{eq:1stmomentHHHd}
\int_{x \in K} \sinh (d_H^s(H,x)) \, d \vol_d^H(x) = 0 \Longleftrightarrow c(K) \in H,
\end{equation}
where $d_H^s(H,x)$ denotes the signed hyperbolic distance of $H$ and $x$; the sign is determined by the orientation of $H$. Here the integral in (\ref{eq:1stmomentHHHd}), denoted by $M_H(K)$, is called the \emph{first moment} of $K$ with respect to $H$.
\end{remark}

We note that the same definition of hyperbolic centroid can also obtained from laws of special relativity theory, as was shown by Gal'perin \cite{Galperin}. In addition, using variants of the properties (i) and (ii) in the beginning of this subsection, an axiomatic definition can also be given to define the weighted centroid of a finite system of weighted points in any space of constant curvature. This was also done by Gal'perin in \cite{Galperin}, where the author showed that the above definition is the only one satisfying these axioms. In particular, this shows that the above approach leads to model independent definitions of centroids of convex bodies in these spaces. For more information on centroids in spaces of constant curvature, the interested reader is referred to \cite{Galperin}.

We finish this subsection with defining centroids in a $d$-dimensional normed space $\MM^d$. To do it, we observe that
\begin{itemize}
\item[(i)] The space $\MM^d$ is a $d$-dimensional linear vector space; i.e. it has a linear structure.
\item[(ii)] By a classical theorem of Haar, every bounded, translation invariant Borel measure on $\Re^d$ is a scalar multiple of $d$-dimensional Lebesgue measure. In other words, every `meaningful' definition of volume in $\MM^d$ is a scalar multiple of $d$-dimensional Euclidean volume (see e.g. \cite{AT, BL, Langi16}).
\end{itemize}

\begin{remark}\label{rem:normedcentroid}
Based on the above observations, for any convex body $K$ in the $d$-dimensional normed space $\MM^d$, the \emph{centroid} $c(K)$ of $K$ is defined as
\begin{equation}\label{eq:centroidMMd}
c(K) = \frac{1}{m} \int_{x \in K} x \, d \, \lambda_d(x),
\end{equation}
where $m=\lambda_d(K)$. This notion satisfies the following normed variant of Remark~\ref{rem:leverruleRed}:\\
For any oriented hyperplane $H$, we have
\begin{equation}\label{eq:1stmomentnormed}
\int_{x \in K} d_{M}^s(H,x) d \, \lambda_d(x) = 0 \Longleftrightarrow c(K) \in H,
\end{equation}
where $d_M^s(H,x)$ denotes the signed normed distance of $H$ and $x$ in $\MM^d$; ; the sign is determined by the orientation of $H$. Here the integral in (\ref{eq:1stmomentnormed}), denoted by $M_H(K)$, is called the \emph{first moment} of $K$ with respect to $H$.
\end{remark}

We note that for any convex body $K$ in any of the spaces above, we have $c(K) \in \inter(K)$.

\subsection{Equilibria in non-Euclidean spaces}\label{subsec:equilibria}

Consider a convex body $K$ in the Euclidean space $\Re^d$. If the point $q \in \bd(K)$ has the property that the hyperplane through $q$ and orthogonal to $q-c(K)$ supports $K$, then we say that $q$ is an \emph{equilibrium point} of $K$. It is well known (see e.g. \cite{DLS, DLV, Langi22}) that if $K$ is smooth, then the equilibrium points of $K$ coincide with the critical points of the Euclidean distance function $x \mapsto d_E(x,c(K))$, measured from $c(K)$ and restricted to $\bd (K)$.
We remark that a convex body $K \subset \Re^3$ is called \emph{smooth} if for any boundary point $x$ of $K$ there is a unique supporting hyperplane of $K$ at $q$; this property coincides with the property that $\bd(K)$ is  a $C^1$-class submanifold of $\Re^d$ (cf. \cite{Schneider}).

We define nondegenerate equilibrium points only for smooth convex bodies. If $K$ is smooth, $q \in \bd(K)$ is an equilibrium point of $K$ with a $C^2$-class neighborhood in $\bd(K)$, and the Hessian of the Euclidean distance function on $\bd(K)$, measured from $c(K)$, has nonzero determinant, we say that $q$ is \emph{nondegenerate}. If $K$ has $C^2$-class boundary, finitely many equilibrium points and all are nondegenerate, we say that $K$ is nondegenerate.
In dimensions $d=2$ and $d=3$, nondegenerate equilibrium points are classified based on the number of negative eigenvalues of the Hessian. More specifically, if $d=2$, and the Hessian of the function $x \mapsto d(x,c(K))$ at the equilibrium point $q$ has $0,1$ negative eigenvalue, then $q$ is called \emph{stable} and \emph{unstable}, respectively. For $d=3$, if the Hessian at $q$ has $0,1,2$ negative eigenvalues, then $q$ is called \emph{stable, saddle-type} and \emph{unstable}, respectively.
Based on the observation that the distance function is a Morse function on $\bd(K)$, and using the fact that if $d=2$ or $d=3$, then the Euler characteristic of $\bd(K)$ coincides with that of $\Sph^1$ and $\Sph^2$, respectively, the following important consequence of the Poincar\'e-Hopf theorem \cite{Hopf} holds.

\begin{corollary}[Poincar\'e-Hopf Theorem]\label{cor:PH}
Let $K$ be a nondegenerate convex body in $\Re^d$ with $d=2$ or $d=3$.
\begin{itemize}
\item[(1)] If $d=2$ and $K$ has $S$ stable and $U$ unstable points, then $S-U=0$.
\item[(2)] If $d=3$ and $K$ has $S$ stable, $H$ saddle-type and $U$ unstable points, then $S-H+U=2$.
\end{itemize}
\end{corollary}

By Corollary~\ref{cor:PH}, to determine the numbers of the different types of equilibrium points of a nondegenerate convex body in $\Re^2$, it is sufficient to know the number of its stable points, and in $\Re^3$ the numbers of its stable and unstable points. Thus, every nondegenerate convex body in $\Re^2$ belongs to a class $\{S\}_E$ consisting of the bodies having $S$ stable and $S$ unstable points, and in $\Re^3$ to a class $\{S,U \}_E$ consisting of the bodies having $S$ stable, $U$ unstable and $S+U-2$ saddle-type points. We remark that the global minimizers and maximizers of the Euclidean distance function are stable and unstable points of the body $K$, and thus, the quantities $S,U$ above are positive integers.

For any convex body in spherical or hyperbolic space, the above concepts of equilibria, nondegenerate equilibria, stable, unstable and saddle points can be naturally modified by using the spherical or hyperbolic distance functions, respectively, and the same holds for the consequence of the Poincar\'e-Hopf theorem in Corollary~\ref{cor:PH}.

For a convex body $K$ in a normed space we follow the same approach. More specifically, if $K$ is a convex body in the normed space $\MM^d$ and $q \in \bd(K)$, we say that $q$ is an equilibrium point of $K$ if there is a hyperplane $H$ supporting $K$ at $q$ that is Birkhoff orthogonal to the segment $[c(K),q]$. We recall that this is equivalent to the property that a homothetic copy of the unit ball $M$ of $\MM^d$, centered at $c(K)$, touches $K$. Thus, if both $K$ and $M$ are smooth, then the equilibrium points of $K$ coincide with the critical points of the normed distance function $x \mapsto d_M(c(K),x)$ restricted to $\bd(K)$.
In the paper we discuss the equilibrium points of convex bodies in planes with strictly convex, smooth norms: in this case for any hyperplane $H$ and point $x \notin H$ there is a unique segment $[x,q]$ with $q \in H$ to which it is Birkoff orthogonal, and for every line $[x,q]$ there is a unique hyperplane through $q$  
that is Birkhoff orthogonal to it.

To define nondegenerate equilibrium points of a convex body $K$ we assume that $M$ as well as $K$ has $C^2$-class boundary, and follow the Euclidean definition. With this additional assumption, we define nondegenerate convex bodies, stable, unstable and saddle-type equilibria as in the Euclidean case. We note that under this condition the Poincar\'e-Hopf theorem can be applied in the same way, showing that Corollary~\ref{cor:PH} holds for any nondegenerate convex body in a normed space whose unit ball has $C^2$-class boundary.

\section{The proof of Theorem~\ref{thm:2d}}\label{sec:2d}

Let $\V^2$ be a plane of constant curvature, or a normed plane with a smooth and strictly convex unit disk. For any two points $p,q \in \V^2$, we denote by $d_V(p,q)$ the distance of $p,q$ in $\V^2$. Let $K$ be a plane convex body in $\V^2$. We suppose for contradiction that $K$ has at most three equilibrium points.

Let $B$ denote a disk in $\V^2$ centered at the centroid $c(K)$ of $K$. By \cite[Lemma 4.4]{BHLL} and since every plane convex body can be approximated arbitrarily well by a plane convex body with piecewise smooth boundary, it follows that if $\V^2 = \Sph^2$, then $K$ is contained in the closed hemisphere centered at $c(K)$.
Setting $S=\bd(B)$, we define the radial function $\rho_K : S \to \Re$ of $K$ as follows: For any $x \in S$, let $L_x$ denote the halfline starting at $c(K)$ and containing $x$, and let $q_x$ denote the intersection point of $L_x$ and $\bd(K)$. Then
\[
\rho_K(x) = d_V(q_x,c(K)).
\]
Observe that the function $\rho_K$ is continuous on $S$.

\begin{lemma}\label{lem:monotonicity}
There are points $x_1,x_2 \in S$ such that the function $\rho_K$ is strictly monotone on both arcs of $S$ with endpoints $x_1, x_2$.
\end{lemma}

\begin{proof}
Let $x_1$ and $x_2$ be an absolute minimum and an absolute maximum point of $\rho_K$, respectively. If $\rho_K(x_1)=\rho_K(x_2)$, then $K$ is a disk, and the statement clearly holds. Hence, assume that $r_1=\rho_K(x_1) < \rho_K(x_2)=r_2$.

Assume that $x$ is a local minimum point of $\rho_K$, and let $D$ be the disk centered at $c(K)$ and satisfying $x \in \bd(D)$. Then, within a neighborhood $V$ of the corresponding point $q_x$ of $\bd(K)$, $K$ contains the set $V \cap D$. Since $D$ is smooth and $K$ is convex, the unique supporting line of $D$ at $q_x$ supports $K$ as well. Thus, $q_x$ is an equilibrium point of $K$. It follows similarly that for every local maximum point $x$ of $\rho_K$, the corresponding point $q_x$ of $\bd(K)$  is an equilibrium point. Hence, it follows that $\rho_K$ is not constant on any nondegenerate arc of $S$, as otherwise $K$ has infinitely many equilibrium points.

Now, consider one of the arcs $S'$ of $S$ with endpoints $x_1, x_2$. If $\rho_K$ is not strictly monotone on $S'$, then there are some points $y_1, y_2 \in S' \setminus \{ x_1, x_2 \}$ such that $x_1, y_1, y_2, x_2$ are in this order on $S'$, and $r_1 \leq \bar{r}_1 = \rho_K(y_1) \geq \rho_K(y_2) \leq \bar{r}_2 < R$. Let $S_1, S_2, S_3$ be the closed arcs of $S'$ with endpoints $x_1$ and $y_1$, $y_1$ and $y_2$, and $y_2$ and $x_2$, respectively.

If $r_1 = \bar{r}_1$, then for any absolute maximum point $z$ of $\rho_K$ on the arc $S_1$, the points $x_1, z, y_1, x_2$ are local extremum points, contradicting our assumption. We obtain similarly that $r_2 \neq \bar{r}_2$.

Consider the case that $\bar{r}_1 = \bar{r}_2$. Let $z_1$ and $z_2$ be absolute minimum and maximum points of $\rho_K$ on the arc $S_2$. If both lie in the interior of $S_2$, then $x_1, x_2, z_1, z_2$ are local extremum points of $\rho_K$; a contradiction. If $\rho_K(z_1)=\bar{r}_1=\bar{r}_2$ then, letting $z_3$ be an absolute minimum point of $\rho_K$ on $S_3$, the points $x_1, y_1, z_3, x_2$ are local extremum points of $\rho_K$; a contradiction. The case that $\rho_K(z_2)=\bar{r}_1=\bar{r}_2$ can be excluded similarly.

We are left with the case $r_1 < \bar{r}_2 < \bar{r}_1 < r_2$. In that case a case analysis similar to the one in the previous paragraph yields the assertion.
\end{proof}

In the following, we call a pair of points $q_1, q_2 \in S$ \emph{antipodal} if the reflection of $q_1$ to $c(K)$ is $q_2$. Since $\rho_K$ is a continuous function, we may apply the Borsuk-Ulam theorem, and obtain the existence of two antipodal points $y_1,y_2 \in S$ such that $\rho_K(y_1)= \rho_K(y_2)$.  Let $S_1'$ and $S_2'$ be the two open arcs of $S$ connecting $y_1$ and $y_2$ such that one of them, say $S_1'$ contains $x_1$. Then, by Lemma~\ref{lem:monotonicity}, for any pair of antipodal points $q_1 \in S_1'$ and $q_2 \in S_2'$, $\rho_K(q_1) < \rho_K(q_2)$. Thus, if $L$ denotes the line containing the collinear points $c(K), y_1, y_2$, and $h : \V^2 \to \V^2$ is the reflection with respect to $c(K)$, then $h(K) \setminus K$ and $K \setminus h(K)$ are strictly separated by $L$. Thus, the first moment $M_L(K)$ of $K$ with respect to $L$ is not zero. This, by Remarks~\ref{rem:leverruleRed}, \ref{rem:leverruleSphd}, \ref{rem:leverruleHHHd}, and \ref{rem:normedcentroid}, contradicts the fact that $L$ passes through $c(K)$.

\section{The proof of Theorem~\ref{thm:3d}}\label{sec:3d}

First we present the proof for the spherical space $\Sph^3$ in Subsection~\ref{subsec:spherical}. Next, in Subsections~\ref{subsec:hyperbolic} and \ref{subsec:normed} we explain how to modify the proof for hyperbolic and normed spaces, respectively.

\subsection{The case $\V^3 = \Sph^3$}\label{subsec:spherical}

Consider the unit sphere $\Sph^3$ of $\Re^4$ centered at the origin $o$. Set the point $v_N = (0,0,0,1)$; we call this point the \emph{North Pole} of $\Sph^3$, and the set of points of $\Sph^3$ with positive $x_4$-coordinates the \emph{Northern Hemisphere} $N^3$ of $\Sph^3$. Let $H_N$ be the hyperplane of $\Re^4$ with equation $\{ x_4 = 1 \}$. The \emph{gnomonic projection} of $\Sph^3$ onto $H_N$ is defined as the map $g_N : N^3 \to H_N$, where the image $g_N(x)$ of $x \in N^3$ is the intersection of $H_N$ and the halfline $\{ \lambda x : \lambda > 0 \}$. Note that $g_N$ is a bijection between $N^3$ and $H_N$ with the property that $X \subset N^3$ is a greatcircle arc in $N^3$ iff its image $g_N(X)$ is a Euclidean segment in $H_N$. Thus, for any spherically convex body $K \subset N^3$, $g_N(K)$ is a convex body in $H_N$, regarded as a $3$-dimensional Euclidean space, and vice versa: if $L$ is a (Euclidean) convex body in $H_N$, then $g^{-1}(L)$ is a (spherical) convex body in $N^3 \subset \Sph^3$.

\begin{lemma}\label{lem:sphericalvolume}
For any point $y \in H_N$, set $|y|_N = |y-v_N|$. Then the following holds.
\begin{enumerate}
\item[(i)] For any point $x \in N^3$, its spherical distance from $v_N \in \Sph^3$ is
\[
d_S(v_N,x)= \arctan |g_N(x)|_N.
\]
\item[(ii)] For any Borel set $X \subset N^3$,
\[
\vol_3^S(X) = \int_{x \in g_N(X)} \frac{1}{\left( 1+ |x|_N^2 \right)^2} \, d \lambda_3(x).
\]
\end{enumerate}
\end{lemma}

Here, (i) is based on an elementary geometric observation and (ii) can be proved by elementary calculus. For more information on the properties of gnonomic projection (as well as for a more general form of (ii)) the interested reader is referred to \cite{BW}. 

We also need the next lemma in our construction. In this, for any point $x \in H_N$, we denote the coordinates of $x$ by $x=(x_1,x_2,x_3,1)$.

\begin{lemma}\label{lem:momentcomp}
Let $H_{3}$ denote the hyperplane of $\Re^4$ with equation $\{x_3 = 0\}$, oriented towards the normal vector $u=(0,0,1,0)$. Let $K \subset N^3$ be a convex body. Then the first moment of $K$ with respect to the hyperplane (greatsphere) $G=H_{3} \cap \Sph^3$ is
\begin{equation}\label{eq:momentcomp}
M_G(K) = \int_{x \in g_N(K)} \frac{x_3}{\left( 1+ |x|_N^2 \right)^{5/2}} d \lambda_3(x).
\end{equation}
\end{lemma}

\begin{proof}
Consider any point $q \in \Sph^3$, and let the angle of $q$ and $u=(0,0,1,0)$ be denoted by $\gamma$. Then the inner product of $q$ and $u$ is $\langle q, u \rangle = \cos \gamma$. On the other hand, we can observe that $d_3^S(q,G)=\frac{\pi}{2} - \gamma$, implying that $\sin d_3^S(q,G) = \cos \gamma = \langle q, u \rangle$.

Now, let $x=(x_1,x_2,x_3,1) \in H_N$. Then
\[
g_N^{-1}(x) = \left( \frac{x_1}{\sqrt{1+|x|_N^2}}, \frac{x_2}{\sqrt{1+|x|_N^2}}, \frac{x_3}{\sqrt{1+|x|_N^2}}, \frac{1}{\sqrt{1+|x|_N^2}} \right) .
\]
Thus, $\langle g_N^{-1}(x) , u \rangle = \frac{x_3}{\sqrt{1+|x|_N^2}}$. Now the assertion follows from the definition of first moment in Remark~\ref{rem:leverruleSphd}, and using the formula for the volume element appearing in (ii) of Lemma~\ref{lem:sphericalvolume}.
\end{proof}

Next, observe that for any spherical ball $B$ in $\Sph^3$ centered at $v_N$, $g_N(B)$ is a Euclidean ball in $H_N$ centered at $v_N$, and vice versa. Thus, in the remaining part of Subsection~\ref{subsec:spherical}, we construct a mono-monostatic convex body in $H_N$ arbitrarily close to a given Euclidean ball $B$ in $H_N$, centered at $v_N$ and of radius $R > 0$. We note that $K \subset N^3$ is mono-monostatic with centroid $c(K)=v_N$ iff $g_N(K)$ is mono-monostatic with respect to $v_N$ (in Euclidean sense). Thus, with a little abuse of notation, we identify $H_N$, equipped with the volume element $\frac{1}{\left( 1+ |x|_N^2 \right)^2} \, d \lambda_3(x)$, with the northern hemisphere $N^3$ of $\Sph^3$. In our computations we regard $v_N$ as the origin of the space, and omit the fourth coordinates of the points. 

In our construction we follow the one in \cite{DLV} for $\Re^3$ with necessary modifications. We set $S$ as the unit sphere $S=(v_N + \Sph^3) \cap H_N$ of $H_N$ centered at the `origin' $v_N$.

The construction proceeds as follow. In the first part, we define a two-parameter family $\mathcal{F}$ of star-shaped bodies $K(c,d)$ such that for any choice of the parameters
\begin{itemize}
\item[(A)] $K(c,d)$ has $C^2$-class boundary;
\item[(B)] $K(c,d)$ has exactly one stable and one unstable point with respect to the `origin' $v_N$ of $H_N$.
\end{itemize}    
In the second part, we show that for any $\varepsilon > 0$, there is a suitable choice of the parameters $c,d$ such that
\begin{itemize}
\item[(C)] $K(c,d)$ has positive Gaussian curvature at every boundary point, implying also that it is convex;
\item[(D)] the (spherical) centroid of $K_m(c,d)$ is $v_N$;
\item[(E)] $d_H(K(c,d), B) < \varepsilon$.
\end{itemize}

We start with the construction of the family $\mathcal{F}$.
                                              
Let us define the function $F_c : [0,1] \to \Re$, 
\begin{equation}\label{eq:defnF}
F_c(x) = \frac{cx^2 + (1-c)(1-x)^2 \cdot \frac{cx}{c+x}}{cx+(1-c)(1-x)^2},
\end{equation}
where $0 < c \leq 1$ (see Fig. \ref{fig:Fc}).

\begin{figure}
\centering
\includegraphics[width=0.5\textwidth]{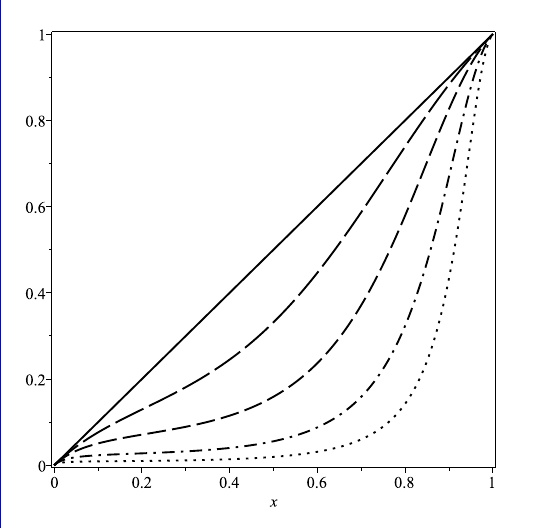}
\caption{Illustration of the function $F_c(x)$ for different values of $c$. Notation: continuous line for $c=1$, long dashed line for $c=0.3$, dashed line for $c=0.1$, dashed-dotted line for $c=0.03$, and dotted line for $c=0.01$.}
\label{fig:Fc}
\end{figure}

The next lemma can be found as \cite[Lemma 1]{DLV}.

\begin{lemma}\label{lem:Fproperties}
For any $c \in (0,1)$, 
\begin{enumerate}
\item[(a)] $F_c$ is smooth.
\item[(b)] $F_c$ is strictly increasing on its domain.
\item[(c)] $F_c(0) = 0$ and $F_c(1) = 1$.
\item[(d)] $(F_c)'_{+}(0)=(F_c)'_{-}(1)=1$.
\end{enumerate}
Furthermore,
\begin{enumerate}
\item[(e)] $F_1(x) = x$ for all $x \in [0,1]$.
\item[(f)] 
$\lim_{c \to 0+0} F_c(x) = 0$ for all $x \in [0,1)$.
\end{enumerate}
\end{lemma}

By Lemma~\ref{lem:Fproperties}, the range of $F_c$ is $[0,1]$ for all $c \in (0,1)$. Next, define $f_c : \left[ - \frac{\pi}{2}, + \frac{\pi}{2} \right] \to \left[ - \frac{\pi}{2}, + \frac{\pi}{2} \right]$ as
\begin{equation}\label{eq:f}
f_c(\theta) = \pi F_c\left( \frac{\theta}{\pi} + \frac{1}{2} \right) - \frac{\pi}{2}.
\end{equation}
Then, for any $0 < c < 1$, $f_c$ is a linear image of $F_c$.

Set $g_c(\theta) = - f_c(-\theta)$. Clearly, the domain and the range of $g_c$ is $\left[ -\frac{\pi}{2}, \frac{\pi}{2} \right]$. In Remark~\ref{rem:elemprop} (see \cite[Remark 2]{DLV}) we summarize a few properties of $f_c$ and $g_c$.

\begin{remark}\label{rem:elemprop}
For any $0 < c \leq 1$, we have  
$f_c \left( \frac{\pi}{2} \right)=g_c \left( \frac{\pi}{2} \right)=\frac{\pi}{2}$, 
$f_c \left( -\frac{\pi}{2} \right)=g_c \left( -\frac{\pi}{2} \right)=-\frac{\pi}{2}$.
\end{remark}

For any $0 < c \leq 1$, $0 \leq \theta \leq 2\pi$ and $\varphi \in \Re$, we set
\begin{equation}\label{eq:a}
a_c(\theta, \varphi) = \frac{ \cos^2 \varphi \cos^2 f_c(\theta)}{ \cos^2 \varphi\cos^2 f_c(\theta) + \sin^2 \varphi \cos^2 g_c(\theta) },
\end{equation}
as illustrated by Fig. \ref{fig:ac}. We observe that if $\cos \varphi \neq 0$ (i.e. if $\varphi$ is not of the form $\frac{\pi}{2}+k\pi$ for some integer $k$), then
$a_c(\theta, \varphi) = \frac{1}{1 + \tan^2 \varphi \frac{\cos^2 g_c(\theta)}{\cos^2 f_c(\theta)} }$.
\begin{figure}[!ht]
\centering
\includegraphics[width=0.48\textwidth]{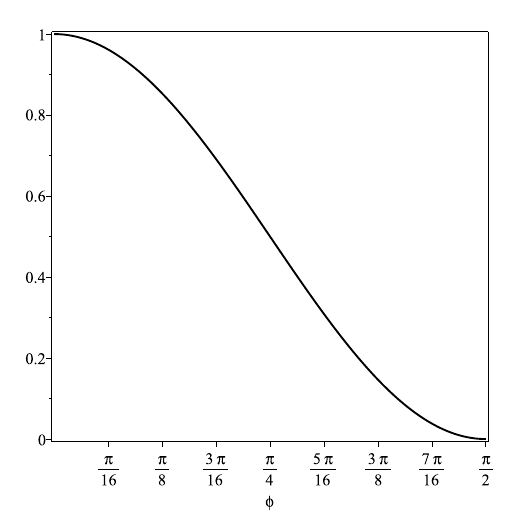}
\includegraphics[width=0.48\textwidth]{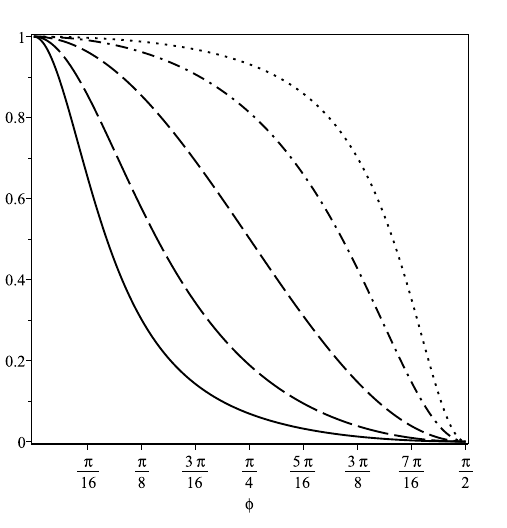}
\caption{Diagram of $a_c(\theta,\varphi)$ versus $\varphi$ for two different values of $c$. If $c=1$ then $a_c=\cos(2\varphi)$ for all values of $\theta$ (left panel), whereas for other values of $c$ and $\theta\neq 0$, it varies. The right panel shows the values of $a_c(\theta,\varphi)$ for $c=0.1$, with $\theta=-0.45\pi$ (continuous line), $\theta=-0.225 \theta$ (long dashed line), $\theta=0$ (dashed line), $\theta=0.225\pi$ (dashed-dotted line), and $\theta=0.45\pi$ (dotted line).}
\label{fig:ac}
\end{figure}

For any unit vector $u = (\cos \theta \cos \varphi, \cos \theta \sin \varphi, \sin \theta ) \in S$, define the function $\rho_c : \S^2 \to \Re$ by
\begin{equation}\label{eq:rho}
\rho_c(u) =  a_c(\theta, \varphi) \sin f_c(\theta) + (1 - a_c(\theta, \varphi)) \sin g_c(\theta),
\end{equation}
(see Fig. \ref{fig:rho}) and set
\begin{equation}\label{eq:R}
R_{c,d}(u) = R \left( 1+ d \rho_c(u) \right).
\end{equation}

\begin{figure}[!ht]
\centering
\includegraphics[width=0.7\textwidth]{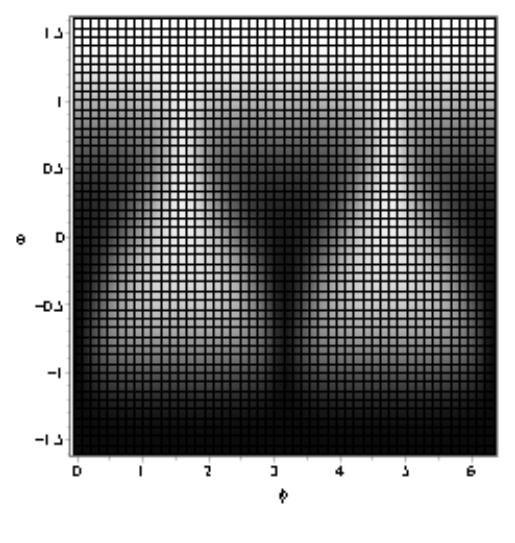}
\caption{Contour plot of $\rho_{0.1}(\theta,\varphi)$. Darker color denotes smaller value of the function.}
\label{fig:rho}
\end{figure}
%

%
%
Observe that by Remark~\ref{rem:elemprop}, both $\rho_c$ and $R_{c,d}$ are well-defined at the two poles, with parameters $\theta = \pm \frac{\pi}{2}$.

Finally, we define the set $K(c,d)$ (see Fig. \ref{fig:K}) as
\[
K(c,d) = \left\{ v_N + \lambda (u-v_N) : u \in S, 0 \leq \lambda \leq R_{c,d}(u) \right\},
\]
and remark that by definition, for any value of the parameters we have $0 \leq a_c(\theta, \varphi) \leq 1$. Thus, $\rho_c(u)$ is a convex combination of $f_c(\theta)$ and $g_c(\theta)$. This, by Lemma~\ref{lem:Fproperties}, implies that $-1 \leq \rho_c(u) \leq 1$ and $R(1-d) \leq R_{c,d}(u) \leq R(1+d)$, implying that $K(c,d)$ exists for all $0 \leq d < 1$. 

\begin{figure}[!ht]
\centering
\includegraphics[width=0.31\textwidth]{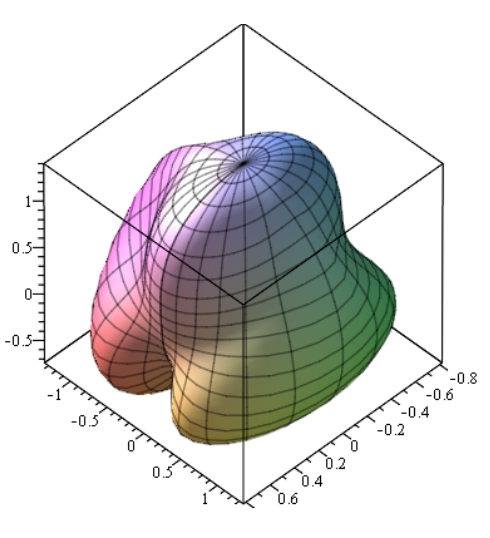}
\includegraphics[width=0.31\textwidth]{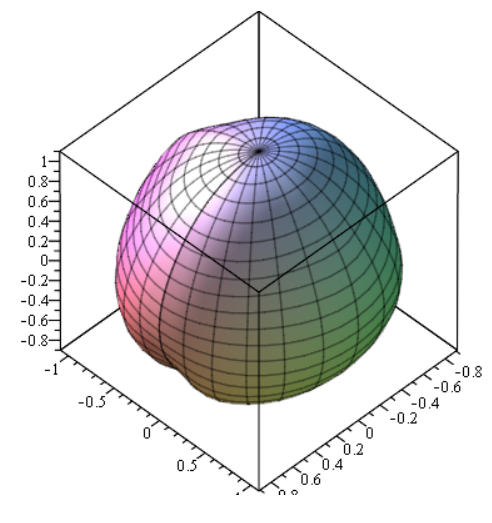}
\includegraphics[width=0.31\textwidth]{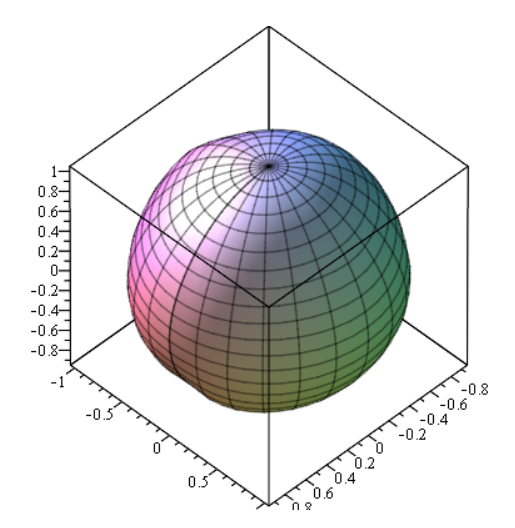}
\caption{The set $K(0.1,d)$ for some values of $d$. Left panel: $d=0.4$, middle panel: $d=0.1$, right panel: $d=0.05$.}
\label{fig:K}
\end{figure}

Next, we prove that the properties described in (A)-(B) are satisfied for $K(c,d)$ for all $0 < d < 1$ and $0 < c \leq 1$.
To show them, we recall \cite[Lemmas 3, 4]{DLV}; we state them using the notation introduced in this subsection.

\begin{lemma}\label{lem:C2}
For any $0 < c < 1$ and $0 < d < 1$, the function $R_{c,d} : S \to \Re$ is $C^2$-class differentiable.
\end{lemma}

\begin{lemma}\label{lem:equilibria}
For any $0 < c < 1$ and $ 0 < d < 1$, the only equilibrium points of $K(c,d)$ with respect to the `origin' $v_N$ are $R_{c,d}(u_N) u_N$ and $R_{c,d}(u_S) u_S$, where $u_N=(0,0,1)$ and $u_S=(0,0,-1)$.
\end{lemma}

Now we prove that suitable elements of $\mathcal{F}$ satisfy the properties (C)-(E).
For this purpose, we set some arbitrary small positive constant $\varepsilon > 0$, and observe that the definition of $\rho_c$ implies that $\max \{ \rho_c(u) : u \in S \} = 1$, and hence, if $0 < d < \varepsilon$, then $d_H(K_n(c,d), R B) < R \varepsilon$. Thus, the assertion of Theorem~\ref{thm:3d} for $\Sph^3$ readily follows from Lemmas~\ref{lem:centering} and \ref{lem:convexity}.

First, in Lemma~\ref{lem:centering} we examine (D). 

\begin{lemma}\label{lem:centering}
There exist constants $0 < c_1 < c_2 < 1$, $0 < d_0 < 1$ and a function $F:(0,d_0) \to [c_1,c_2]$ such that for any $d \in (0,d_0)$
the (spherical) centroid of $K(F(d),d)$ is $v_N$.
\end{lemma}

\begin{proof}
We start the proof with the observation that $K(c,d)$ is axially symmetric to the translate of the $x_3$-axis of $\Re^4$ by $v_N$. Hence, the centroid of $K(c,d)$ is of the form $(0,0,z(c,d))$, implying that by Lemma~\ref{lem:momentcomp}, the (spherical) centroid of $K(c,d)$ is $v_N$ iff the (spherical) first moment of $K(c,d)$ with respect to the plane of $H_N$ with equation $\{ x_3 = 0 \}$ is $0$. We denote this first moment by $M_3(K(c,d))$.

Changing to $3$-dimensional spherical coordinates and by Lemma~\ref{lem:momentcomp}, $M_3(K(c,d))$ can be written as
\begin{multline*}
M_3(K(c,d)) = \int_{x \in K} \frac{x_3}{\left( 1+ |x|_N^2 \right)^{5/2}} d \lambda_3(x) =\\
= 
\int_{0}^{2\pi}\int_{-\frac{\pi}{2}}^{\frac{\pi}{2}}\int_{0}^{R(1+d\rho_c)} \frac{r\sin(\theta)}{(1+r)^{5/2}}\cdot r^2\cos(\theta) \,dr\,d\theta\,d\varphi,
\end{multline*}
yielding that
\begin{equation}\label{eq:firstmoment}
M_3(K(c,d)) = \int_{0}^{2\pi}\int_{-\frac{\pi}{2}}^{\frac{\pi}{2}}\left(-\frac{1}{3}\cdot\frac{3R^2(1+d\rho_c)^2+2}{((R^2(1+d\rho_c)^2+1)^\frac{3}{2}}+\frac{2}{3}\right)\cdot\cos(\theta)\sin(\theta)\,d\theta\,d\varphi .
\end{equation}

We note that by the properties of the trigonometric functions, $M_3(K_m(c,0))=0$ for all $0 < c \leq 1$.

Observe that the integrand in (\ref{eq:firstmoment}), as well as its partial derivative with respect to $d$, is continuous on the domain $0 < c \leq 1$ and $0 \leq d < R$. Hence, the same holds for $M_3(K(c,d))$.
Note that $\rho_1(\theta, \varphi) = \sin(\theta)$. Thus,
\begin{multline*}
M_3(K(1,d)) =\\
=2\pi\int_{-\frac{\pi}{2}}^{\frac{\pi}{2}}\left(-\frac{1}{3}\cdot\frac{3R^2(1+d\sin(\theta))^2+2}{(R^2(1+d\sin(\theta))^2+1)^\frac{3}{2}}+\frac{2}{3}\right)\cdot\cos(\theta)\sin(\theta)\,d\theta
\end{multline*}
Applying the substitution $\sin \theta = x$, we obtain that
\begin{equation}\label{eq:firstmomentc1_2}
M_3(K(1,d)) = 2\pi\int_{-1}^{1}\left(-\frac{1}{3}\cdot\frac{3R^2(1+d\cdot x)^2+2}{(R^2(1+d\cdot x)^2+1)^\frac{3}{2}}+\frac{2}{3}\right)\cdot x\,dx
\end{equation}

First, we consider the value of $M_3(K_m(c,d))$ near the point $(c,d)=(1,0)$. Recall that $M_3(K(1,0))=0$.
On the other hand, by Leibniz's integral rule, we have
\[
\left.\frac{\partial M_{3}(K_m(1,d))}{\partial d}\right|_{d=0} =2\pi\int_{-1}^{1}\frac{x^2R^4}{(R^2+1)^\frac{5}{2}}\,dx = \frac{4\pi}{3}\cdot \frac{R^4}{(R^2+1)^\frac{5}{2}} >0.
\]
Thus, there is some value $d_0> 0$ such that $K(1,d)) > 0$ for all values of $d$ with $0 < d \leq d_0$.

Next, we consider the value of $M_3(K(c,d))$ near the point $(c,d)=(0,0)$. Recall that $M_3(K(c,0))=0$ for any $0 < c \leq 1$.
On the other hand, by Leibniz's integral rule, for any $0 < c \leq 1$, we have
\begin{multline*}
\left.\frac{\partial M_{3}(K(c,d))}{\partial d}\right|_{d=0} = \\
=\int_{0}^{2\pi}\int_{-\frac{\pi}{2}}^{\frac{\pi}{2}}\left.\frac{\partial }{\partial d}\left(-\frac{1}{3}\cdot\frac{3R^2(1+d\rho_c)^2+2}{(R^2(1+d\rho_c)^2+1)^\frac{3}{2}}+\frac{2}{3}\right)\right|_{d=0}\cdot\cos(\theta)\sin(\theta)\,d\theta\,d\varphi
\end{multline*}
implying that
\[
\left.\frac{\partial M_{3}(K(c,d))}{\partial d}\right|_{d=0} = \frac{2\pi R^4}{(R^2+1)^\frac{5}{2}} \int_{-\frac{\pi}{2}}^{\frac{\pi}{2}}\rho_c\cdot\cos(\theta)\sin(\theta)\,d\theta.
\]

We recall from the proof of Lemma 5 of \cite{DLV} that
\[
\lim_{c\to 0^+}\rho_c(\theta,\varphi)=\frac{(\frac{\pi}{2}-\theta)^4\sin^2 \varphi-(\frac{\pi}{2}+\theta)^4\cos^2 \varphi}{(\frac{\pi}{2}+\theta)^4\cos^2 \varphi+(\frac{\pi}{2}-\theta)^4\sin^2 \varphi}.
\]
Let us define the above expression as $\rho_0(\theta,\varphi)$.
Since $| \rho_c(\theta,\varphi) | \leq 1$ for all values $0 < c \leq 1$, $0 \leq  \varphi \leq 2\pi$, $-\frac{\pi}{2} < \theta < \frac{\pi}{2}$, and the constant function $(\theta,\varphi) \mapsto 1$ is integrable on $[0,2\pi] \times \left[-\frac{\pi}{2}, \frac{\pi}{2} \right]$, then, by Lebesgue's dominated convergence theorem \cite{Evans}, for any $0 \leq d < R$ we have
\begin{multline}\label{eq:firstmomentfinal}
\lim_{c\to 0^+} \left.\frac{\partial M_{3}(K(c,d))}{\partial d}\right|_{d=0} =\\
= \frac{2\pi R^4}{(R^2+1)^\frac{5}{2}} \int_{-\frac{\pi}{2}}^{\frac{\pi}{2}}\rho_c\cdot\cos(\theta)\sin(\theta)\,d\theta = \frac{2\pi R^4}{(R^2+1)^\frac{5}{2}}
\cdot (-4.6335890\ldots) < 0.
\end{multline}
Thus, there is some value $c_0 > 0$ such that $\left.\frac{\partial M_{3}(K(c_0,d))}{\partial d}\right|_{d=0} < 0$, implying that there is some value $d_0'>0$ such that $M_{3}(K(c_0,d))$ for all $0 < d \leq d_0'$. Combining this result with our investigation near the point $(c,d)=(1,0)$, the assertion follows.
\end{proof}

We note that the integral in (\ref{eq:firstmomentfinal}) is also computed in \cite{DLV}, where the authors incorrectly stated that its value is $-2.3168\ldots$; nevertheless, this typo in \cite{DLV} does not influence the correctness of their proof.

Next, we deal with Property (C). To do it, following \cite{Schneider} we denote by $\mathcal{C}^2_+$ the family of convex bodies whose boundary is $C^2$-class differentiable at every point, and has strictly positive Gaussian curvature everywhere.

\begin{lemma}\label{lem:convexity}
Let $[c_1,c_2]$ be defined as in Lemma~\ref{lem:centering}. Then there is a value $0 < d^{\star} < 1$ such that for any 
$0 < d < d^{\star}$ and $c \in [c_1,c_2]$, we have $K(c,d) \in \mathcal{C}^2_+$. 
\end{lemma}

This lemma appeared as \cite[Lemma 6]{DLV}, and hence, we omit its proof. As a conclusion, we note that for any given $\varepsilon > 0$, if $0 < d < d^{\star}$ is sufficiently small, $K(F(d),d)$ satisfies the conditions (A)-(E) .

\subsection{The case $\V^3 = \HHH^3$}\label{subsec:hyperbolic}

In the proof we use the projective ball model of the hyperbolic space $\HHH^3$. In this model, the space is represented as the interior of the unit ball $\B^3$ of the Euclidean space $\Re^3$; here hyperbolic lines and planes are represented by the intersections of the set $\inter (\B^3)$ with Euclidean lines and planes, respectively.
Thus, a hyperbolic set in this model is a convex body iff it is represented by a convex body in the model. We note that whereas the model is not conformal, if $L$ is a hyperbolic line passing through the center $o$ of the model, then a hyperbolic plane $H$ is perpendicular to $L$ iff the segment representing $L$ in the model is perpendicular in the Euclidean sence to the circular disk representing $H$.

Our next lemma collects well known properties of the model.

\begin{lemma}\label{lem:hypvolume}
For any point $x \in \Re^3$, let $|x|$ denote the Euclidean distance of $x$ from the origin $o$. Then the following holds.
\begin{enumerate}
\item[(i)] For any point $x \in \HHH^3$, its hyperbolic distance from $o$ is
\[
d_H(o,x)= \arctanh |x| = \frac{1}{2} \ln \frac{1+|x|}{1-|x|}.
\]
\item[(ii)] For any Borel set $X \subset \HHH^3$ its hyperbolic volume is
\[
\mu_3(X) = \int_{x \in g_N(X)} \frac{1}{\left( 1- |x|^2 \right)^2} \, d \lambda_3(x)
\]
\end{enumerate}
\end{lemma}

Lemma~\ref{lem:momentcomphyp} is the hyperbolic counterpart of Lemma~\ref{lem:momentcomp}.

\begin{lemma}\label{lem:momentcomphyp}
Let $H$ denote the hyperbolic plane with equation $\{x_3 = 0\}$, oriented towards the normal vector $u=(0,0,1)$. Let $K \subset \HHH^3$ be a convex body. Then the first moment of $K$ with respect to $H$ is
\begin{equation}\label{eq:momentcomphyp}
M_{H}(K) = \int_{x \in K} \frac{x_3}{\left( 1- |x|^2 \right)^{5/2}} d \lambda_3(x).
\end{equation}
\end{lemma}

\begin{proof}
Consider any point $x=(x_1,x_2,x_3) \in \HHH^3$, and let $y=(x_1,x_2,0)$. By (i) of Lemma~\ref{lem:hypvolume}, $\cosh d_H(o,x)= \frac{1}{\sqrt{1-|x|^2}}$ and $\cosh d_H(o,y) = \frac{1}{\sqrt{1-|y|^2}}$.  By the hyperbolic Pythagorean Theorem, $\cosh d_H(y,x) = \frac{\cosh d_H(o,x)}{\cosh d_H(o,y)} = \frac{\sqrt{1-|y|^2}}{\sqrt{1-|x|^2}}$. Thus, $\sinh d_H(y,x) = \sqrt{\cosh^2 d_H(o,x)-1} = \frac{x_3}{\sqrt{1-|x|^2}}$. Combining this formula with the volume element appearing in (ii) of Lemma~\ref{lem:hypvolume}, the assertion follows.
\end{proof}

Consider a ball $B$ of radius $0 < R < 1$ centered at $o$. This ball represents a hyperbolic ball of radius $\arctanh R$. In our construction we find a slight deformation $K$ of $B$ such that its (hyperbolic) centroid is $o$, and it has only two equilibrium points. The construction, using the same family $\mathcal{F}$ defined in Subsection~\ref{subsec:spherical} in $H_N$, is entirely analogous to the one in Subsection~\ref{subsec:spherical}, setting $S= \Sph^3$, and thus, we omit it.

\subsection{The case when $\V^3$ is a rotationally symmetric $3$-dimensional normed space}\label{subsec:normed}

Let $M$ be the unit ball of $\V^3$. Without loss of generality, we may assume that $M$ is symmetric to the $x_3$-axis of the space. We make use of the ambient Euclidean structure of the space, and denote by $\rho_M(u) : \Sph^2 \to \Re$ the \emph{radial function} of $M$ \cite{Schneider}, defined as
\[
\rho_M(u) = \sup \{ \lambda: \lambda u \in M \}.
\]
We follow the construction in Subsections~\ref{subsec:spherical} and \ref{subsec:hyperbolic} with one exception: we define the function $R_{c,d}(u)$ for any $u \in \Sph^2$ by
\begin{equation}\label{eq:Rnormed}
R_{c,d}(u) = R \rho_M(u)(1 + d \rho_c(u)).
\end{equation}
Using this definition, we let
\[
K(c,d) = \left\{ \lambda u : u \in \Sph^2, 0 \leq \lambda \leq R_{c,d}(u) \right\}.
\]
Thus, for small values of $d>0$, $K(c,d)$ is a slight deformation of the ball $R M$. The proof follows the one in Subsection~\ref{subsec:spherical}. We present the proofs of only two lemmas, as the rest of the lemmas in that subsection can be proved analogously to their counterparts in Subsection~\ref{subsec:spherical}.

\begin{lemma}\label{lem:equilibrianormed}
For any $0 < c < 1$ and $ 0 < d < 1$, the only equilibrium points of $K(c,d)$ with respect to $o$ are $R_{c,d}(u_N) u_N$ and $R_{c,d}(u_S) u_S$, where $u_N=(0,0,1) \in \Sph^2$ and $u_S=(0,0,-1) \in \Sph^2$.
\end{lemma}

\begin{proof}
Note that by the symmetry of $K(c,d)$, the points $R_{c,d}(u_N) u_N$ and $R_{c,d}(u_S) u_S$ are equilibrium points of $K(c,d)$.
Assume that for some $-\frac{\pi}{2} < \theta_0 < \frac{\pi}{2}$ and setting
$u_0 = (\cos \theta_0 \cos \varphi_0, \cos \theta_0 \sin \varphi_0, \sin \theta_0) \in \Sph^2$,
$K(c,d)$ has an equilibrium point at $R_{c,d}(u_0) u_0 \in \bd(K(c,d))$. Then the tangent plane $H$ of $K(c,d)$ at this point is Birkhoff orthogonal to $u_0$. In other words, $H$ is tangent at $R_{c,d}(u_0) u_0$ not only to $K(c,d)$, but also to $R (1 + d \rho_c(u_0) M$, i.e. the homothetic copy of $M$ centered at $o$ and containing  $R_{c,d}(u_0) u_0$ in its boundary. Thus, we have
$\left(R_{c,d}\right)'_{\theta}(u_0) = R(1 + d \rho_c(u_0) \left( \rho_M \right)'_{\theta}(u_0)$, and $\left(R_{c,d}\right)'_{\varphi}(u_0) = R(1 + d \rho_c(u_0) \left( \rho_M \right)'_{\varphi}(u_0)$.
By the definition of $R_{c,d}$, from this we obtain that $(\rho_c)'_{\theta} (u_0) = (\rho_c)'_{\varphi} (u_0) = 0$. Hence, by the Euclidean result in \cite[Lemma 4]{DLV}, we have a contradiction.
\end{proof}

\begin{lemma}\label{lem:centeringnormed}
There exist constants $0 < c_1 < c_2 < 1$, $0 < d_0 < 1$ and a function $F:(0,d_0) \to [c_1,c_2]$ such that for any $d \in (0,d_0)$,
$c(K(F(d),d))=o$.
\end{lemma}

\begin{proof}
We start the proof with the observation that as $K(c,d)$ is axially symmetric to the $x_3$-axis, its centroid is of the form $c(K(c,d))=(0,0,z(c,d))$. Thus, $c(K)=o$ if and only if the first moment of $K$ with respect to the plane $H_3$ with equation $\{ x_3 = 0 \}$ is $0$. We denote this first moment by $M_3(K)$.

Changing to spherical coordinates, up to a constant factor independent of $c,d$ and setting $x=(x_1,x_2,x_3)$, $M(K_m(c,d))$ can be written as
\begin{multline*}
M_3(K(c,d)) = \int_{x \in K} x_3  \, d \lambda_3(x) =\\
= 
\int_{0}^{2\pi}\int_{-\frac{\pi}{2}}^{\frac{\pi}{2}}\int_{0}^{R \rho_M(1 + d \rho_c)} r^3 \sin \theta \cos \theta \,dr\,d\theta\,d\varphi,
\end{multline*}
implying that
\begin{equation}\label{eq:firstmomentnormed}
M_3(K(c,d)) = \int_{0}^{2\pi}\int_{-\frac{\pi}{2}}^{\frac{\pi}{2}} \frac{1}{4} \left( R \rho_M(1 + d \rho_c) \right)^3 \sin \theta \cos \theta \,d\theta\,d\varphi .
\end{equation}

We note that as $M$ is $o$-symmetric, $M_3(K(c,0))=0$ for all $0 < c \leq 1$. Furthermore, from this it also follows that
\begin{equation}\label{eq:firstmomentnormedmod}
M_3(K(c,d)) = d \int_{0}^{2\pi}\int_{-\frac{\pi}{2}}^{\frac{\pi}{2}}  R^4 \rho_M^4 \left( \rho_c + \frac{3}{2} d \rho_c^2 + d^2 \rho_c^3 + \frac{d^3}{4} \rho_c^4\right) \sin \theta \cos \theta \,d\theta\,d\varphi .
\end{equation}

Observe that the integrand in (\ref{eq:firstmomentnormed}), as well as its partial derivative with respect to $d$, is continuous on the domain $0 < c \leq 1$ and $0 \leq d < 1$. Hence, the same holds for $M_3(K(c,d))$.
Note that $\rho_1(\theta, \varphi) = \sin(\theta)$. Thus,
\begin{multline*}
M_3(K(1,d)) =\\
= d \int_{0}^{2\pi}\int_{-\frac{\pi}{2}}^{\frac{\pi}{2}}  R^4 \rho_M^4 \left( \sin^2 \theta + \frac{3}{2} d \sin^3 \theta + d^2 \sin^4 \theta + \frac{d^3}{4} \sin^5 \theta\right) \cos \theta \,d\theta\,d\varphi
\end{multline*}
For any $-\frac{\pi}{2} \leq \theta \leq \frac{\pi}{2}$, set $\hat{\rho}_M(\theta) = \int_{0}^{2\pi} \rho_M^4(\theta,\varphi) \, d \varphi$. Then, as $M$ is $o$-symmetric, we have that $\hat{\rho}_M(\theta)$ is an even function. Thus, as the function $\theta \mapsto \sin^k \theta$ is even if $k$ is even and odd if $k$ is odd, it follows that
\begin{equation}\label{eq:firstmomentc1_2normed}
M_3(K_m(1,d)) = d \int_{-\frac{\pi}{2}}^{\frac{\pi}{2}}  2 R^4 \hat{\rho}_M \left( 2 \sin^2 \theta  + d^2 \sin^4 \theta \right) \cos \theta \,d\theta,
\end{equation}
which is positive for all values $0<d \leq 1$. In particular, $M_3(K(c,d))$ is positive in a neighborhood of the point $(c,d)=(1,0)$.

Next, we consider the value of $M_3(K(c,d))$ near the point $(c,d)=(0,0)$. Recall that $M_3(K(c,0))=0$ for any $0 < c \leq 1$. Thus, to show that $M_3(K(c,d))$ is negative for sufficiently small values of $c$ and $d$, it is sufficient to show that
\[
\bar{M}(c,d)= \int_{0}^{2\pi}\int_{-\frac{\pi}{2}}^{\frac{\pi}{2}}  \rho_M^4 \rho_c \sin \theta \cos \theta \,d\theta\,d\varphi < 0
\]
in such a region. Now, using the notation $\rho_0$ introduced in Subsection~\ref{subsec:spherical} and applying Lebesgue's dominated convergence theorem, we obtain that
\[
\lim_{c \to 0^+} \bar{M}(c,d) = \int_{0}^{2\pi}\int_{-\frac{\pi}{2}}^{\frac{\pi}{2}}  \rho_M^4 \rho_0 \sin \theta \cos \theta \,d\theta\,d\varphi .
\]
Now, since $\rho_M$ is independent of $\varphi$, it follows that
\[
\lim_{c \to 0^+} \bar{M}(c,d) = - \int_{-\frac{\pi}{2}}^{\frac{\pi}{2}} \rho_M^4 \frac{8 \theta \pi^2}{\pi^2+4 \theta^2} \sin \theta \cos \theta \,d\theta,
\]
which, as the integrand is an even function, is negative.
Combining this result with our investigation near the point $(c,d)=(1,0)$, the assertion follows.
\end{proof}

\end{document}